\documentclass{amsart}

\usepackage{amsmath}
\usepackage{amssymb}
\usepackage{mathtools}
\usepackage{braket}
\usepackage{prettyref}

\usepackage{subcaption}
\usepackage{tikz}

\usepackage[numbers]{natbib}
\usepackage{url}
\usepackage[hidelinks]{hyperref}

\theoremstyle{plain}
\newtheorem{theorem}{Theorem}[section]
\newtheorem{proposition}[theorem]{Proposition}
\newtheorem{lemma}[theorem]{Lemma}
\newtheorem{corollary}[theorem]{Corollary}
\newtheorem{problem}[theorem]{Problem}
\newtheorem{observation}[theorem]{Observation}

\newtheorem{fact}[theorem]{Fact}

\theoremstyle{definition}
\newtheorem{definition}[theorem]{Definition}
\newtheorem{example}[theorem]{Example}

\newrefformat{app}{Appendix \ref{#1}}
\newrefformat{cor}{Corollary \ref{#1}}
\newrefformat{exa}{Example \ref{#1}}
\newrefformat{fact}{Fact \ref{#1}}
\newrefformat{obs}{Observation \ref{#1}}
\newrefformat{prob}{Problem \ref{#1}}
\newrefformat{prop}{Proposition \ref{#1}}
\newrefformat{eq}{Equation \eqref{#1}}

\newcommand{\paren}[1]{\left(#1\right)}
\newcommand{\abs}[1]{\left\lvert#1\right\rvert}

\newcommand{\sqbra}[1]{\left[#1\right]}

\newcommand{\eqw}{\mathrel{=_{w}}}
\newcommand{\neqw}{\mathrel{\neq_{w}}}
\newcommand{\tow}{\mathrel{\rightarrow_{w}}}
\newcommand{\toli}{\mathrel{\rightarrow_{\mathrm{li}}}}
\newcommand{\tail}{\mathrel{\sim_{\mathrm{tail}}}}

\newcommand{\TrInv}[1]{\mathrm{Tr}_{#1}}
\newcommand{\ShiftTr}{\mathcal{S}}
\newcommand{\id}{\operatorname{id}}

\newcommand{\Nat}{\mathbb{N}}

\newcommand{\Terms}{\mathcal{T}}
\newcommand{\Alg}{\mathcal{A}}
\newcommand{\NTerms}{\mathcal{N}}
\newcommand{\Traces}{\paren{\Nat\times\Nat}^{\Nat}}
\newcommand{\QTraces}{\paren{\Nat\times\Nat}^{\Nat}/{\tail}}

\newcommand{\B}{\mathbf{B}}
\newcommand{\C}{\mathbf{C}}
\newcommand{\I}{\mathbf{I}}
\newcommand{\M}{\mathbf{M}}
\renewcommand{\S}{\mathbf{S}}
\newcommand{\W}{\mathbf{W}}
\newcommand{\Y}{\mathbf{Y}}

\newcommand{\BI}{\mathbf{BI}}
\newcommand{\BM}{\mathbf{BM}}
\newcommand{\BMI}{\mathbf{BMI}}
\newcommand{\BWI}{\mathbf{BWI}}

\title{Finite \texorpdfstring{$\BM$}{BM}--algebras and fixed-point combinators}
\author{Takuma Imamura}
\address{Independent Researcher, Kyoto, Japan.}
\email{imamura.takuma.66s@kyoto-u.jp}

\begin{document}

\begin{abstract}
We study algebras for the $\mathbf{BM}$--, $\mathbf{BI}$--, and $\mathbf{BMI}$--fragments of combinatory logic, where $\mathbf{B} xyz \to_{w} x\left(yz\right)$, $\mathbf{M} x \to_{w} xx$ and $\mathbf{I} x \to_{w} x$.
We prove that every finite $\mathbf{BM}$--algebra $\mathcal{A}$ contains an element $Y$ satisfying $Yx = x\left(Yx\right)$ for all $x\in A$.
Moreover, a single closed $\mathbf{BM}$--term represents such an element in every finite $\mathbf{BM}$--algebra.
On the other hand, we show that there exists an infinite $\mathbf{BMI}$--algebra with no such element.
Consequently, the variety of $\mathbf{BM}$--algebras is not generated by finite $\mathbf{BM}$--algebras.
We also characterise finite $\mathbf{BI}$--algebras as finite semigroups with two distinguished left units $B$ and $I$, and finite $\mathbf{BMI}$--algebras as finite idempotent semigroups with three distinguished left units $B$, $M$ and $I$.
\end{abstract}

\maketitle

\section*{Introduction}

Consider the fragment of combinatory logic that consists of the following combinators:
\begin{align*}
\paren{\paren{\B x}y}z &\tow x\paren{yz}, \\
\M x &\tow xx, \\
\I x &\tow x.
\end{align*}
In his famous book on combinatory logic, Smullyan raised the question of whether there exists a fixed-point combinator in that system.

\begin{problem}[\citet{Smullyan_2000}]\label{prob:Smullyan-problem}
Is there a closed term $Y$ in the $\BMI$--fragment such that $Yx \eqw x \paren{Yx}$ for a fresh variable symbol $x$?
\end{problem}

To solve this problem, it is natural to consider its algebraic formulation.
The idea is that if one thoroughly understands the structure of a model satisfying the axioms of the $\BMI$--fragment (viewed as an equational theory), the original problem can be solved.

\begin{definition}
Let $\mathcal{A} = \paren{A, \ast}$ be a magma, i.e., a non-empty set $A$ with a binary operation $\ast$.
An element $\xi$ of $A$ is called
\begin{enumerate}
\item a \emph{$\B$--element} if $\paren{\paren{\xi x}y}z = x\paren{yz}$ for all $x,y,z\in A$.
\item an \emph{$\M$--element} if $\xi x = xx$ for all $x\in A$.
\item an \emph{$\I$--element} if $\xi x = x$ for all $x \in A$.
\item a \emph{$\Y$--element} if $\xi x = x\paren{\xi x}$ for all $x \in A$.
\end{enumerate}
A magma equipped with a $\B$--element and an $\M$--element is called a \emph{$\BM$--algebra}.
The terms `$\B$--algebra', `$\BI$--algebra' and `$\BMI$--algebra' are defined in a similar fashion.
\end{definition}

Then, the original problem is reduced to considering the following algebraic problem.

\begin{problem}[An algebraic version of \prettyref{prob:Smullyan-problem}]\label{prob:algebraic-Smullyan-problem}
Does every $\BMI$--algebra have a $\Y$--element?
\end{problem}

Recently, \prettyref{prob:Smullyan-problem} was solved negatively in \citep{Imamura_2026} \emph{without explicitly solving \prettyref{prob:algebraic-Smullyan-problem}}.
The proof involves the construction of an invariant $\TrInv{x}$ for terms that separates $Yx$ and $x\paren{Yx}$.

\begin{fact}[\citep{Imamura_2026}]\label{fact:no-periodic-point-combinator}
Let $x$ be a variable symbol and $r \geq 1$.
For every $x$--free term $Y$ in the $\BMI$--fragment,
\[
Yx \neqw x^{r}\paren{Yx},
\]
where $x^{r}{-}$ denotes the $r$--fold application of $x$, i.e., $x^{r} \paren{Yx} = x\paren{x\paren{\cdots x \paren{Yx} \cdots}}$.
\end{fact}

In fact, as we will see, the proof of \prettyref{fact:no-periodic-point-combinator} also yields the negative answer to \prettyref{prob:algebraic-Smullyan-problem}.

\begin{theorem}[\prettyref{thm:infinite-bmi-wo-y}]\label{thm:intro-infinite-bmi-wo-y}
There exists an infinite $\BMI$--algebra $\mathcal{A}$ that has no $\Y$--element.
\end{theorem}

Prior to the complete solution of \prettyref{prob:Smullyan-problem}, the author attempted to (negatively) solve the problem by constructing a finite $\BMI$--algebra that lacks a $\Y$--element; however, this attempt was unsuccessful.
For instance, we obtained the following results using a SAT solver.

\begin{observation}\label{obs:sat-observation}
Every $\BM$--algebra of cardinality $\leq 12$ has a $\Y$--element.
(See also \prettyref{app:code-availability}.)
\end{observation}

Furthermore, during an exhaustive search for $\BM$-- and $\BMI$--algebras, the author observed that the vast majority of them have somewhat `simple' structures, which in turn force the existence of a $\Y$--element.
This naturally gives rise to the following problem.

\begin{problem}[A finite analogue of \prettyref{prob:algebraic-Smullyan-problem}]\label{prob:finite-algebraic-Smullyan-problem}
Does every finite $\BM$--algebra have a $\Y$--element?
\end{problem}

In this paper, we provide an \emph{affirmative} answer to this problem: every finite $\BM$--algebra has a $\Y$--element.
This explains why an infinite algebraic system was required to resolve \prettyref{prob:algebraic-Smullyan-problem}.

The paper is organised as follows.
In \prettyref{sec:trace-invariant-revisited}, we first recall the definitions of the infinite $\BMI$--algebra $\Alg_{BMI}$ and the trace invariant $\TrInv{x}$, which were introduced in \citep{Imamura_2026} to prove \prettyref{fact:no-periodic-point-combinator}.
We then prove \prettyref{thm:intro-infinite-bmi-wo-y} by showing that $\Alg_{BMI}$ has no $\Y$--element.

In \prettyref{sec:finite-bmi-algebras}, we characterise finite $\BMI$--algebras as finite idempotent semigroups whose distinguished elements $B$, $M$, and $I$ are left units.
As a consequence, every finite $\BMI$--algebra has a $\Y$--element.

In \prettyref{sec:finite-bm-algebras}, we investigate finite
$\BM$--algebras.
We first construct a finite $\BI$--algebra as a quotient of an idempotent left ideal of a finite $\B$--algebra.
We use this construction to show that, in finite $\BM$--algebras, the eventual image of each left translation coincides with the set of its fixed points.
We then prove that every non-standard fixed-point combinator gives a $\Y$--element in every finite $\BM$--algebra, thereby answering \prettyref{prob:finite-algebraic-Smullyan-problem}.
Together with the infinite example in \prettyref{sec:trace-invariant-revisited}, this shows that the variety of $\BM$--algebras is not generated by finite $\BM$--algebras.

In \prettyref{sec:open-problems}, we consider several related open problems concerning the relationship between the permutative effects in the sense of \citet{MW_1991} and the existence of fixed-point combinators, as well as the existence of Turing-type fixed-point combinators in the $\BWI$--fragment.

\section{Preliminaries}

We refer to \citet{CHS_1972, Barendregt_2012} for general references of combinatory logic and lambda calculus.

Throughout this paper, we assume that application is left-associative, and we omit parentheses where appropriate: $uvw := \paren{\paren{uv}w}$.
Similarly, we assume that the binary operation of a magma is left-associative: $u\ast v\ast w := \paren{u \ast v} \ast w$.
The operator symbol $\ast$ is omitted unless there is a risk of confusion.

For terms $t$ and $u$ and an $n\in\Nat$, we define the term $t^{n} u$ as follows:
\begin{align*}
t^{0} u &= u, \\
t^{n+1} u &= t\paren{t^{n}u}.
\end{align*}
Similarly, for elements $a$ and $b$ of a magma and an $n\in\Nat$, define the element $a^n b \in A$ as follows:
\begin{align*}
a^{0} b &= b, \\
a^{n+1} b &= a\paren{a^{n}b} \paren{= a\ast\paren{a^{n}b}}.
\end{align*}

We denote the set of all terms in the $\BMI$--fragment by $\Terms_{\BMI}$, the set of all normal terms by $\NTerms_{\BMI}$.
The one-step weak reduction, its reflexive transitive closure, and its equivalence closure are denoted by $\tow$, $\tow^{\ast}$, and $\eqw$, respectively.
The one-step leftmost-innermost reduction is denoted by $\toli$.

The set $\Terms_{\BMI}$ forms the free magma generated by the set of all variable and constant symbols, whose magma operation is given by application.
The quotient algebra $\Terms_{\BMI} / {\eqw}$ forms a free $\BMI$--algebra generated by the set of all variable symbols.
The same applies to the $\BM$--fragment.

Let $\mathcal{A} = \paren{A, \ast, B, M}$ be a $\BM$--algebra.
For a closed term $t$ that may contain elements of $A$ as constant symbols, we denote the interpretation of $t$ in $\mathcal{A}$ by $t^{\mathcal{A}}$.

\section{An infinite \texorpdfstring{$\BMI$}{BMI}--algebra without \texorpdfstring{$\Y$}{Y}}\label{sec:trace-invariant-revisited}

In this section, we show that the $\BMI$--algebra $\Alg_{\BMI}$, introduced in \citep{Imamura_2026}, provides an example of an infinite $\BMI$--algebra that lacks a $\Y$--element.

\subsection{Definition of the \texorpdfstring{$\BMI$}{BMI}--algebra \texorpdfstring{$\Alg_{\BMI}$}{A\_BMI}}

Throughout this section, we fix a variable symbol $x$.
We first recall the definition of the $\BMI$--algebra $\Alg_{BMI}$ and the homomorphism $\TrInv{x}\colon \Terms_{\BMI} \to \Alg_{\BMI}$.

\begin{definition}[Tail equivalence]
For $\alpha,\beta \in \Traces$, we say that $\alpha$ and $\beta$ are \emph{tail equivalent} (denoted by $\alpha\tail\beta$) if there exist $i,j\in\Nat$ such that $\alpha_{i+k} = \beta_{j+k}$ holds for all $k \in \Nat$.
\end{definition}

\begin{definition}[Shift operators]
Let $c\in\Nat$.
Define the map $S_{c}\colon \Traces \to \Traces$ by
\[
\ShiftTr_{c} \paren{\paren{p_{i}, q_{i}}}_{i\in\Nat} := \paren{\paren{p_{i} + c, q_{i}}}_{i\in\Nat}.
\]
Since $S_{c}$ preserves $\tail$, it induces the map
\[
\ShiftTr_{c}\colon \QTraces \to \QTraces,\qquad \sqbra{\alpha} \mapsto \sqbra{S_{c}\alpha}.
\]
\end{definition}

\begin{definition}[Traces]
Let $u \in \Terms_{\BMI}$ and suppose that $u$ does not have a weak normal form.
Then there exists a (unique) infinite leftmost-innermost reduction sequence $u=u_{0}\toli u_{1} \toli \cdots$ (see \citep[Lemma 1.1]{Imamura_2026}).
For $i \in \Nat$, let
\[
\Gamma_{i}\sqbra{\M A_{i}} \toli \Gamma_{i}\sqbra{A_{i}A_{i}}
\]
be the $i$--th contraction of $M$ in the sequence.
Define the \emph{$x$--trace} $\tau_{x}\paren{u}$ of $u$ by
\[
\tau_{x}\paren{u} := \paren{\paren{p_{i}, q_{i}}}_{i\in\Nat} \in \Traces,
\]
where
\begin{align*}
    p_{i} &:= \text{the number of occurrences of $x$ in $\Gamma_{i}\sqbra{\M A_{i}}$ strictly before $\M A_{i}$}, \\
    q_{i} &:= \abs{A_{i}}_{x}.
\end{align*}
\end{definition}

\begin{definition}[Trace invariants]
For a term $u \in \Terms_{\BMI}$, the \emph{$x$--trace invariant} $\TrInv{x}\paren{u}$ is defined by
\[
\TrInv{x}\paren{u} :=
\begin{cases}
    N \in \NTerms_{\BMI}, & u \text{ has a normal form }N, \\
    \sqbra{\tau_{x}\paren{u}} \in \QTraces, & \text{otherwise}.
\end{cases}
\]
\end{definition}

\begin{definition}[Algebra of traces]
Let
\[
\Alg_{\BMI} := \NTerms_{\BMI} \sqcup \paren{\QTraces}.
\]
Define the binary operation
\[
\star_{x}\colon \Alg_{\BMI} \times \Alg_{\BMI} \to \Alg_{\BMI}
\]
by
\[
a \ast_{x} b := \begin{cases}
    a, & a \in \QTraces, \\
    \ShiftTr_{\abs{a}_{x}} b, & a \in \NTerms_{\BMI},\ b\in \QTraces, \\
    \TrInv{x}\paren{ab}, & a,b\in\NTerms_{\BMI}.
\end{cases}
\]
\end{definition}

\begin{fact}[Theorem 2.10 in \citep{Imamura_2026}]
The structure $\paren{\Alg_{\BMI}, \star_{x}, \B, \M, \I}$ forms a $\BMI$--algebra.
\end{fact}

\begin{fact}[Theorem 2.9 in \citep{Imamura_2026}]
The map $\TrInv{x}\colon \Terms_{\BMI} \to \Alg_{\BMI}$ is a homomorphism, i.e., $\TrInv{x}\paren{uv} = \TrInv{x}\paren{u} \star_{x} \TrInv{x}\paren{v}$ for all $u,v\in\Terms_{\BMI}$.
\end{fact}

\begin{fact}[Theorem 2.11 in \citep{Imamura_2026}]
The homomorphism $\TrInv{x}\colon \Terms_{\BMI} \to \Alg_{\BMI}$ is invariant under weak equivalence $\eqw$: for any $u,v\in\Terms_{\BMI}$, if $u\eqw v$, then $\TrInv{x}\paren{u} = \TrInv{x}\paren{v}$.
\end{fact}

In other words, the trace invariant can be considered as a homomorphism from $\Terms_{\BMI} / {\eqw}$ to $\Alg_{\BMI}$.
In general, an invariant for the $\BMI$--fragment is precisely a pair of a $\BMI$--algebra $\mathcal{A}$ and a homomorphism from $\Alg_{\BMI} / {\eqw}$ to $\mathcal{A}$.

The proof of \prettyref{fact:no-periodic-point-combinator} (\citep[Theorem 3.1]{Imamura_2026}) actually proves the following stronger statement:

\begin{fact}[Proof of Theorem 3.1 in \citep{Imamura_2026}]\label{fact:trace-separation}
Let $r \geq 1$.
For every $x$--free term $Y \in \Terms_{\BMI}$,
\[
\TrInv{x}\paren{Yx} \neq \TrInv{x}\paren{x^r \paren{Yx}}.
\]
\end{fact}

\subsection{Absence of \texorpdfstring{$\Y$}{Y}--elements}

Now, let us prove that no $\Y$--element exists in this $\BMI$--algebra.
The proof proceeds by showing that if $Y \star_{x} \TrInv{x}\paren{\M\M}$ yields a fixed point for $L_{\TrInv{x}\paren{\M\M}}$, then $Y \star_{x} x$ cannot be a fixed point for the left translation $L_{x}$.

\begin{theorem}\label{thm:infinite-bmi-wo-y}
Let $r\geq 1$.
For every $Y \in \Alg_{\BMI}$, there exists an $a \in \Alg_{\BMI}$ such that
\[
Y \star_{x} a \neq a^{r} \star_{x} \paren{Y \star_{x} a}.
\]
In particular, the $\BMI$--algebra $\Alg_{\BMI}$ has no $\Y$--element.
\end{theorem}

\begin{proof}
Let $\omega := \TrInv{x}\paren{\M\M} \in \QTraces$.
Since $x$ does not appear in the $\toli$--sequence of $\M\M$, the $x$--trace of $\M\M$ is
\[
\tau_{x}\paren{\M\M} = \paren{\paren{0,0}}_{i\in\Nat}.
\]
Thus, for every $c > 0$,
\[
\omega \neq \ShiftTr_{c} \omega.
\]

Now, let $Y \in \Alg_{BMI}$ and suppose that
\[
Y\star_{x} \omega = \omega^{r} \star_{x} \paren{Y\star_{x} \omega}.
\]
If $Y \in \NTerms_{\BMI}$, then
\begin{align*}
\omega &= \omega \star_{x} \paren{\omega^{r-1} \star_{x} \paren{Y \star_{x} \omega}} \\
&= \omega^{r} \star_{x} \paren{Y \star_{x} \omega} \\
&= Y \star_{x} \omega \\
&= S_{\abs{Y}_{x}} \omega.
\end{align*}
It follows that $\abs{Y}_{x} = 0$.
By \prettyref{fact:trace-separation},
\[
Y \star_{x} x \neq x^{r} \star_{x} \paren{Y \star_{x} x}.
\]
On the other hand, if $Y \in \QTraces$, then
\begin{align*}
Y \star_{x} x &= Y \\
&= Y \star_{x} \omega \\
&= \omega^{r} \star_{x} \paren{Y \star_{x} \omega} \\
&= \omega \star_{x} \paren{\omega^{r-1} \star_{x} \paren{Y \star_{x} \omega}} \\
&= \omega, \\
x^{r} \star_{x} \paren{Yx} &= x^{r} \star_{x} \omega \\
&= \ShiftTr_{r} \omega.
\end{align*}
Since $\omega \neq \ShiftTr_{r} \omega$, we have
\[
Y \star_{x} x \neq x^{r} \star_{x} \paren{Y \star_{x} x}.
\]

Thus, no element $Y$ of $\mathcal{A}_{BMI}$ satisfies $Y \star_{x} \omega = \omega^{r} \star_{x} \paren{Y \star_{x} \omega}$ and $Y \star_{x} x = x^{r} \star_{x} \paren{Y \star_{x} x}$ simultaneously.
\end{proof}

The proof also shows that the witness $a$ can be chosen from the set $a\in \set{x, \omega}$.
This is a much stronger property than the non-existence of a $\Y$--element.

\begin{proposition}\label{prop:cardinality-of-trace-algebra}
The $\BMI$--algebra $\Alg_{\BMI}$ is of cardinality $2^{\aleph_{0}}$.
\end{proposition}

\begin{proof}
Define the map $f\colon \Nat^{\Nat} \to \QTraces$ by
\[
f\paren{\alpha} := \paren{\paren{\prod_{j \leq i}p_{j}^{\alpha_{j}}, i}}_{i\in \Nat},
\]
where $p_{j}$ denotes the $j$--th prime.
The map $f$ is injective, so
\[
2^{\aleph_{0}} = \abs{\Nat^{\Nat}} \leq \abs{\QTraces}.
\]
On the other hand, we have
\[
\abs{\QTraces} \leq \abs{\Traces} = 2^{\aleph_{0}}
\]
by the axiom of choice.
Hence
\[
\abs{\Alg_{\BMI}} = \abs{\NTerms_{\BMI}} + \abs{\QTraces} = \aleph_{0} + 2^{\aleph_{0}} = 2^{\aleph_{0}}. \qedhere
\]
\end{proof}

\begin{corollary}
For every infinite cardinality $\kappa$, there exists a $\BMI$--algebra of cardinality $\kappa$ that has no $\Y$--element.
\end{corollary}

\begin{proof}
Immediate from \prettyref{prop:cardinality-of-trace-algebra} and the L\"owenheim--Skolem theorem.
\end{proof}

\section{Finite \texorpdfstring{$\BMI$}{BMI}--algebras}\label{sec:finite-bmi-algebras}

In this section, we investigate the structures of finite $\BMI$--algebras.
As we shall see later, a finite $\BMI$--algebra can be characterised as a finite idempotent semigroup with three left units.
It follows that every finite $\BMI$--algebra possesses a $\Y$--element.

\subsection{Finite \texorpdfstring{$\BI$}{BI}--algebras as semigroups}

We begin by examining the structure of finite $\BI$--algebras.

\begin{proposition}\label{prop:finite-bi-is-semigrp-w-left-unit}
Every finite $\BI$--algebra is a semigroup with $B$ and $I$ as its left units.
\end{proposition}

\begin{proof}
Let $\mathcal{A} = \paren{A, \ast, B, I}$ be a finite $\BI$--algebra.
For each $a \in A$, let $L_{a} \colon A \to A$ be its left translation:
\[
L_{a}\paren{x} := a x.
\]
Consider the (finite) set of all those maps:
\[
\mathcal{L} := \set{L_{a} | a \in A}.
\]

We first note that an element $a \in A$ is left unit if and only if $L_{a} = \id_{A}$.
The element $I$ is a left unit by definition.
We only need to show that the operator of $\mathcal{A}$ is associative and that $B$ is another left unit.

Define the map $\varphi \colon \mathcal{L} \to \mathcal{L}$ by
\[
\varphi\paren{f} := L_{f\paren{I}}.
\]

Let $a \in A$.
Since $B a I x = a \paren{I x} = a x$ for all $x \in A$, we have
\[
\varphi\paren{L_{B a}} = L_{L_{B a}\paren{I}} = L_{B a I} = L_{a}.
\]
Thus $\varphi$ is surjective, and therefore bijective by $\abs{A} < \aleph_{0}$.

For every $x \in A$, since $B I I x = I \paren{I x} = I x = x = I I I x$, we have
\[
\varphi^{2}\paren{L_{B}} = \varphi\paren{L_{BI}} = L_{B I I} = L_{I I I} = \varphi\paren{L_{II}} = \varphi^{2}\paren{L_{I}}.
\]
Hence $L_{B} = L_{I} = \mathrm{id}_{A}$ by the injectivity of $\varphi$.
This means that $B$ is a left unit.

The $B$--axiom now implies that the associativity $\paren{xy}z = \paren{\paren{Bx}y}z = x\paren{yz}$.
\end{proof}

In fact, the converse also holds.
A finite $\BI$--algebra is therefore precisely a finite semigroup with two specified left units.

\begin{proposition}\label{prop:semigrp-w-left-unit-is-bi}
Let $\mathcal{A}=\paren{A, \ast}$ be a semigroup with left units $B$ and $I$.
(Here, $B$ and $I$ are not necessarily distinct.)
Then $\paren{A, \ast, B, I}$ is a $\BI$--algebra.
\end{proposition}

\begin{proof}
Let $x, y, z \in A$.
Since $\ast$ is associative and $B$ is a left unit, we have
\[
\paren{\paren{B x} y} z = B x y z = x y z = x \paren{y z}. \qedhere
\]
\end{proof}

\subsection{Finite \texorpdfstring{$\BMI$}{BMI}--algebras as idempotent semigroups}

\begin{theorem}\label{thm:finite-bmi-is-idempotent}
Every finite $\BMI$--algebra is an idempotent semigroup with $B$, $M$ and $I$ as its left units.
\end{theorem}

\begin{proof}
Let $\mathcal{A} = \paren{A, \ast, B, M, I}$ be a finite $\BMI$--algebra.
By \prettyref{prop:finite-bi-is-semigrp-w-left-unit}, $\paren{A, \ast}$ is a semigroup with $B$ and $I$ as left units.
Let $x \in A$.
By the associativity,
\[
xx = Mx = M\paren{Ix} = \paren{MI}x = II x = x. \qedhere
\]
\end{proof}

The converse of \prettyref{thm:finite-bmi-is-idempotent} also holds.
Therefore, a finite $\BMI$--algebra is precisely a finite idempotent semigroup with three specified left units.

\begin{theorem}\label{thm:idempotent-semigrp-w-left-unit-is-bmi}
Let $\paren{A, \ast}$ be an idempotent semigroup with left units $B$, $M$ and $I$.
Then $\paren{A, \ast, B, M, I}$ is a $\BMI$--algebra.
\end{theorem}

\begin{proof}
The structure $\paren{A, \ast, B, I}$ is a $\BI$--algebra by \prettyref{prop:semigrp-w-left-unit-is-bi}.
Moreover, for every $x \in A$, since $M$ is a left unit and $x$ is idempotent,
\[
M x = x = x x.
\]
Hence $M$ is an $\M$--element.
\end{proof}

\begin{corollary}\label{cor:finite-bmi-has-y}
Every finite $\BMI$--algebra has a $\Y$--element.
\end{corollary}

\begin{proof}
Let $\mathcal{A} = \paren{A,\ast,B,M,I}$ be a finite $\BMI$--algebra.
Then $\paren{A, \ast}$ is an idempotent semigroup with $B$ and $I$ as its left units by \prettyref{thm:finite-bmi-is-idempotent}.
Now, let $Y$ be any left unit of $\mathcal{A}$.
For every $x \in A$, since $x$ is idempotent,
\[
Yx = x = xx = x\paren{Yx}.
\]
Hence $Y$ is a $\Y$--element.
\end{proof}

This explains why it is not possible to prove the absence of a fixed-point combinator by finding a finite-valued invariant.

\begin{corollary}
There is no pair $\paren{\mathcal{A}, f}$ of a finite $\BMI$--algebra $\mathcal{A}$ and a homomorphism $f\colon \Terms_{\BMI} / {\eqw} \to \mathcal{A}$ such that $f\paren{Yx} \neq f\paren{x\paren{Yx}}$ for all cosed terms $Y$ and variables $x$.
\end{corollary}

\section{Finite \texorpdfstring{$\BM$}{BM}--algebras}\label{sec:finite-bm-algebras}

In this section, we prove that every finite $\BM$--algebra has a $\Y$--element.
Unlike in the case of $\BMI$--algebras, $\BM$--algebras have entirely different structures.

\begin{proposition}\label{tab:bm-not-semigrp-exists}
There exists a finite $\BM$--algebra that is neither idempotent nor associative.
\end{proposition}

\begin{proof}
\begin{table}
\caption{\label{tab:bm-not-semigrp}A $\BM$--algebra that is not a semigroup.
}

\begin{tabular}{c|ccc|cccc}
$\ast$ & $0$ & $1$ & $2$ & $B$ & $M$ & $I$ & $Y$ \tabularnewline
\hline 
$0$ & $0$ & $0$ & $0$ & $\checkmark$ &  &  & $\checkmark$ \tabularnewline
$1$ & $0$ & $0$ & $1$ &  & $\checkmark$ &  &  \tabularnewline
$2$ & $0$ & $0$ & $1$ &  & $\checkmark$ &  &  \tabularnewline
\end{tabular}
\end{table}
Define the binary operation $\ast$ on $A_{3} := \set{0, 1, 2}$ by \prettyref{tab:bm-not-semigrp}.
The magma $\mathcal{A}:=\paren{A_{3}, \ast, B, M}$ forms a $\BM$--algebra, where $B=0$ and $M \in \set{1, 2}$.
However, $\mathcal{A}$ is neither idempotent nor associative:
\begin{align*}
2 \cdot 2 &= 1 \neq 2, \\
\paren{2 \cdot 2} \cdot 2 &= 1 \cdot 2 \\
&= 1, \\
2 \cdot \paren{2 \cdot 2} &= 2 \cdot 1 \\
&= 0.\qedhere
\end{align*}
\end{proof}

The results of the previous section rely on the fact that every finite $\BI$--algebra is a semigroup with $B$ as its left unit (\prettyref{prop:finite-bi-is-semigrp-w-left-unit}), which cannot be directly applied to $\BM$--algebras.
Instead, we reduce the problem to \prettyref{prop:finite-bi-is-semigrp-w-left-unit} by considering a certain idempotent left ideal of a finite $\B$--algebra.

\subsection{Left ideals \texorpdfstring{$A_{\infty}$}{A\_infty} of \texorpdfstring{$\B$}{B}--algebras}

We consider the left ideal $A_{\infty}$ of a $\B$--algebra.
In the finite case, we show that this left ideal is idempotent and contains the eventual image of every left translation.

\begin{definition}[cf.~\citet{Statman_2011,Ikebuchi_2020}]
Let $\mathcal{A} = \paren{A, \ast, B}$ be a $\B$--algebra.
Define another binary operation $\bullet\colon A\times A\to A$ by
\[
a \bullet b := Bab.
\]
\end{definition}

\begin{lemma}\label{lem:composition-homomorphism}
Let $\mathcal{A} = \paren{A, \ast, B}$ be a $\B$--algebra.
The set of left translations $\mathcal{L} := \set{L_{a} | a\in A}$ forms a semigroup with respect to composition.
Moreover, the assignment $a \mapsto L_{a}$ gives a homomorphism from $\paren{A, \bullet}$ to $\paren{\mathcal{L}, \circ}$.
\end{lemma}

\begin{proof}
For all $L_{a}, L_{b} \in \mathcal{L}$, the B--identity $Babx = a\paren{bx}$ gives
\begin{equation}\label{eq:B-composition}\tag{$\dagger$}
L_{a} \circ L_{b} = L_{Bab} = L_{a\bullet b} \in \mathcal{L}.
\end{equation}
Hence $\mathcal{L}$ is a sub-semigroup of the semigroup $A^{A}$.
\prettyref{eq:B-composition} also shows that $a \mapsto L_{a}$ is a semigroup homomorphism.
\end{proof}

\begin{definition}
Let $\mathcal{A} = \paren{A, \ast}$ be a magma.
For $i \in \Nat$, define the subset $A_{i}$ of $A$ by
\begin{align*}
A_{0} &:= A, \\
A_{i+1} &:= A_{i}A_{i}.
\end{align*}
Put $A_{\infty} := \bigcap_{i\in\Nat} A_{i}$.
\end{definition}

\begin{lemma}
Let $\mathcal{A} = \paren{A, \ast, B}$ be a $\B$--algebra.
\begin{enumerate}
\item For each $i\in\Nat$, $A_{i+1} \subseteq A_{i}$.
\item For each $i\in\Nat$, $A_{i}$ is a left ideal of $\mathcal{A}$.
\item $A_{\infty}$ is a left ideal of $\mathcal{A}$.
\end{enumerate}
\end{lemma}

\begin{proof}
The property (3) follows from (1) and (2).
Let us verify (1) and (2) by induction on $i\in\Nat$.
The case where $i=0$ is trivial.
For $i > 0$, let $a \in A$ and $x\in A_{i}$.
Write $x$ as $x = yz$ for $y, z\in A_{i-1}$.
Then
\[
ax = a\paren{yz} = Bayz \in \paren{AA_{i-1}} A_{i-1}.
\]
By the induction hypothesis, we have
\[
\paren{AA_{i-1}} A_{i-1} \subseteq A_{i-1} A_{i-1} = A_{i}.
\]
Since $a$ and $x$ were arbitrary, it follows that $AA_{i} \subseteq A_{i}$.
Furthermore,
\[
A_{i} = A_{i-1}A_{i-1} \subseteq A A_{i-1} \subseteq A_{i-1}.\qedhere
\]
\end{proof}

\begin{corollary}\label{cor:idempotent-submagma}
Let $\mathcal{A} = \paren{A, \ast, B}$ be a finite $\B$--algebra.
Then $A_{\infty} A_{\infty} = A_{\infty}$.
\end{corollary}

\begin{proof}
Since $A$ is finite, the sequence $\paren{A_{i}}_{i\in\Nat}$ saturates, i.e., $A_{i} = A_{\infty}$ for all sufficiently large $i$.
Hence $A_{\infty} A_{\infty} = A_{i} A_{i} = A_{i+1} = A_{\infty}$.
\end{proof}

\begin{definition}[Eventual images]
Let $\mathcal{A} = \paren{A, \ast}$ be a magma.
For each $a \in A$, define the \emph{evential image} of $L_{a}$ by
\[
L_{a}^{\infty}\sqbra{A} := \bigcap_{i \in \Nat} L_{a}^{i}\sqbra{A},
\]
where $L_{a}^{i}$ is the $i$--fold composition of the left translation $L_{a}$.
\end{definition}

\begin{lemma}
Let $\mathcal{A} = \paren{A, \ast}$ be a magma.
\begin{enumerate}
\item $L_{a}^{i}\sqbra{A} \supseteq L_{a}^{i+1}\sqbra{A}$ for all $i\in\Nat$.
\item If $A$ is finite, $L_{a}^{i}\sqbra{A} = L_{a}^{\infty}\sqbra{A}$ for all sufficiently large $i\in\Nat$.
\item If $A$ is finite, $L_{a}\restriction L_{a}^{\infty}\sqbra{A} \colon L_{a}^{\infty}\sqbra{A}\to L_{a}^{\infty}\sqbra{A}$ is bijective.
\end{enumerate}
\end{lemma}

\begin{proof}
Since $L_{a}\sqbra{A}\subseteq A$, we have $L_{a}^{i+1}\sqbra{A}\subseteq L_{a}^{i}\sqbra{A}$ for every $i\in\Nat$.
If $A$ is finite, this descending chain eventually stabilises, i.e. $L_{a}^{i}\sqbra{A} = L_{a}^{\infty}\sqbra{A}$ for all sufficiently large $i\in\Nat$.
Then $L_{a}\sqbra{L_{a}^{\infty}\sqbra{A}} = L_{a}\sqbra{L_{a}^{i}\sqbra{A}} = L^{i+1}\sqbra{A} = L_{a}^{\infty}\sqbra{A}$.
Thus $L_{a} \restriction L_{a}^{\infty}\sqbra{A}$ is surjective, and therefore, bijective by $\abs{A} < \aleph_{0}$.
\end{proof}

\begin{lemma}\label{lem:Ainf-ideal-includes-stable-sets}
Let $\mathcal{A} = \paren{A, \ast, B}$ be a finite $\B$--algebra.
Then $\bigcup_{a\in A} L_{a}^{\infty}\sqbra{A} \subseteq A_{\infty}$.
\end{lemma}

\begin{proof}
Let $a \in A$.
For $i\in\Nat$, define $a_{i} \in A_{i}$ by
\begin{align*}
a_{0} &:= a, \\
a_{i+1} &:= Ba_{i}a_{i} \paren{= a_{i} \circ a_{i}}.
\end{align*}
By applying \prettyref{lem:composition-homomorphism} repeatedly, we have
\[
L_{a_{i}} = L_{a}^{2^{i}}.
\]
Since the map $L_{a}$ is a bijection on $L_{a}^{\infty}\sqbra{A}$, the map $L_{a_{i}}$ is likewise a bijection on $L_{a}^{\infty}\sqbra{A}$.
So
\[
L_{u_{i}}\sqbra{L_{a}^{\infty}\sqbra{A}} = L_{a}^{\infty}\sqbra{A},
\]

Now, we show that $L_{a}^{\infty}\sqbra{A} \subseteq A_{i}$ by induction on $i$.
The case where $i=0$ is trivial.
For $i>0$, since $L_{a}^{\infty}\sqbra{A} \subseteq A_{i}$ by the induction hypothesis,
\[
L_{a}^{\infty}\sqbra{A} = a_{i} L_{a}^{\infty}\sqbra{A} \subseteq A_{i} A_{i} = A_{i+1}. \qedhere
\]
\end{proof}

\subsection{Extensional equivalence}

We introduce the extensional equivalence $\sim$ on $A$ relative to $A_{\infty}$.
For finite $\B$--algebras, the resulting relation is a congruence on $A_{\infty}$, and every equivalence class has a representative in $A_{\infty}$.

\begin{definition}[Extensional equivalences]
Let $\mathcal{A} = \paren{A, \ast}$ be a magma.
For $n\in\Nat$, define the equivalence relation $\sim_{n}$ on $A$ by
\[
a \sim_{n} b \iff a x_{1} \cdots x_{n} = b x_{1} \cdots x_{n}\qquad\text{for all }x_{1},\ldots,x_{n} \in A_{\infty}.
\]
Equivalently, the relations $\sim_{0}, \sim_{1}, \ldots$ can also be defined by induction on $n$:
\begin{align*}
a \sim_{0} b &\iff a = b, \\
a \sim_{n+1} b &\iff ax \sim_{n} bx \qquad\text{for all }x\in A_{\infty}.
\end{align*}
\end{definition}

\begin{definition}[$\beta$--map]
Let $\mathcal{A} = \paren{A, \ast, B}$ be a $\B$--algebra.
Define the map $\beta\colon A/\sim_{1} \to A/\sim_{2}$ by
\[
\beta \sqbra{a} := \sqbra{L_{B}\paren{a}} = \sqbra{Ba}.
\]
\end{definition}

\begin{lemma}\label{lem:extensional-equivalences-and-beta}
Let $\mathcal{A} = \paren{A, \ast, B}$ be a finite $\B$--algebra.
All of ${\sim_{1}}, {\sim_{2}}, \ldots$ are equivalent to each other.
Moreover, the map $\beta$ is a bijection on $A/{\sim}$.
\end{lemma}

\begin{proof}
It suffices to show that ${\sim_{1}} = {\sim_{2}}$, as the general case can be derived from this case by induction.
Let $a, b\in A$. Since $A_{\infty}A_{\infty} = A_{\infty}$ by \prettyref{cor:idempotent-submagma}, we have
\begin{align*}
a \sim_{1} b &\iff ax = bx \qquad\text{for all }x\in A_{\infty} \\
&\iff a\paren{xy} = b\paren{xy} \qquad\text{for all }x,y\in A_{\infty} \\
&\iff Ba xy = Bb xy \qquad\text{for all }x,y\in A_{\infty} \\
&\iff Ba \sim_{2} Bb.
\end{align*}
Hence the map $\beta$ is injective, and therefore, bijective by $\abs{A} < \aleph_{0}$.
Hence
\[
\abs{A / {\sim_{1}}} = \abs{A / {\sim_{2}}}.
\]
It follows that ${\sim_{1}} = {\sim_{2}}$ by $\abs{A} < \aleph_{0}$.
\end{proof}

For a finite $\B$--algebra and $i \geq 1$, we omit the subscript $i$ of $\sim_{i}$.

\begin{lemma}
Let $\mathcal{A} = \paren{A, \ast, B}$ be a finite $\B$--algebra.
For every $a \in A$, there exists an $a_{\infty} \in A_{\infty}$ such that $a \sim a_{\infty}$.
\end{lemma}

\begin{proof}
By \prettyref{lem:Ainf-ideal-includes-stable-sets}, we have $L^{\infty}_{B}\sqbra{A} \subseteq A_{\infty}$.
Since $A$ is finite, $L^{\infty}_{B}\sqbra{A} = L^{m}_{B}\sqbra{A}$ holds for some $m \in \Nat$.

Let $a_{0} \in A$.
Since $\beta$ is surjective by \prettyref{lem:extensional-equivalences-and-beta}, there exist $a_{1}, a_{2}, \ldots, a_{m} \in A$ such that $\sqbra{a_{i}} = \beta \sqbra{a_{i+1}}$ for all $0 \leq i < m$.
Hence $\sqbra{a_{0}} = \beta^{m} \sqbra{a_{m}} = \sqbra{L_{B}^{m}\paren{a_{m}}}$ and $L_{B}^{m}\paren{a_{m}} \in A_{\infty}$.
\end{proof}

\begin{lemma}\label{lem:Ainf-represents-A}
Let $\mathcal{A} = \paren{A, \ast, B}$ be a finite $\B$--algebra.
The equivalence relation $\sim$ is a congruence on the subalgebra $A_{\infty}$.
\end{lemma}

\begin{proof}
Let $a, a', b\in A_{\infty}$ and suppose that $a \sim a'$.
Since $a \sim_{2} a'$ by \prettyref{lem:extensional-equivalences-and-beta},
\[
a b x = a' b x \qquad \text{for all }x\in A_{\infty}.
\]
Hence $ab \sim_{1} ab'$.

Next, let $a, b, b' \in A_{\infty}$ and suppose that $b \sim b'$.
Since the map $\beta\colon A/{\sim_{1}} \to A/{\sim_{2}}$ is surjective by \prettyref{lem:extensional-equivalences-and-beta}, one can find an $a' \in A$ such that
\[
a \sim_{2} Ba'.
\]
Then, for every $x \in A_{\infty}$,
\begin{align*}
a b x &= \paren{Ba'} b x \\
&= a' \paren{bx} \\
&= a' \paren{b' x} \\
&= \paren{Ba'} b' x \\
&= a b' x.
\end{align*}
Hence $a b \sim a b'$.
\end{proof}

\subsection{Quotient algebra \texorpdfstring{$A_{\infty}/{\sim}$}{A\_infty/sim}}

We show that the quotient $\mathcal{Q} := A_{\infty}/{\sim}$ is a semigroup with a left unit and that $\beta$ is the identity on this quotient.
For finite $\BM$--algebras, we use this structure to identify the eventual image of each left translation with its set of fixed points.

\begin{definition}
Let $\mathcal{A} = \paren{A, \ast, B}$ be a finite $\B$--algebra.
Denote the quotient magma of $A_{\infty}$ modulo $\sim$ by $\mathcal{Q}$.
\end{definition}

\begin{theorem}\label{thm:Q-is-BI-algebra}
Let $\mathcal{A} = \paren{A, \ast, B}$ be a finite $\B$--algebra.
The magma $\mathcal{Q}$ has $\B$-- and $\I$--elements.
Moreover, the map $\beta$ is an identity map on $\mathcal{Q}$.
Furthermore, if $\mathcal{A}$ has an $\M$--element, then so does $\mathcal{Q}$.
\end{theorem}

\begin{proof}
By \prettyref{lem:Ainf-represents-A}, there exists a $B_{\infty} \in A_{\infty}$ such that $B_{\infty} \sim B$.
Since $B_{\infty} \sim_{1} B$,
\[
B_{\infty} x y z = B x y z = x \paren{y z} \qquad \text{for all }x,y,z\in A_{\infty}.
\]
Hence $\sqbra{B_{\infty}} \in A_{\infty} / {\sim}$ is a $\B$--element of $\mathcal{Q}$.

The left translation $L_{\sqbra{B_{\infty}}} \colon A_{\infty}/{\sim} \to A_{\infty}/{\sim}$ by $\sqbra{B_{\infty}}$ coincides with $\beta$ on $\mathcal{Q}$:
\[
L_{\sqbra{B_{\infty}}} \sqbra{x} = \sqbra{B_{\infty}x} = \sqbra{Bx} = \beta\sqbra{x}\qquad\text{for all }x\in A_{\infty}.
\]
Hence $L_{\sqbra{B_{\infty}}}$ is bijective by \prettyref{lem:extensional-equivalences-and-beta}.
Since $A_{\infty}/{\sim}$ is finite, there exists an $n \geq$ such that $L_{\sqbra{B_{\infty}}}^{n} = \id_{A_{\infty}/{\sim}}$.
Let $I \in A_{\infty}/{\sim}$ be any element obtained by inserting parentheses in an arbitrary way into
\[
\underbrace{\sqbra{B_{\infty}} \bullet \cdots \bullet \sqbra{B_{\infty}}}_{n}.
\]
By \prettyref{lem:composition-homomorphism}, $L_{I} = L_{\sqbra{B_{\infty}}}^{n} = \id_{A_{\infty}/{\sim}}$, i.e., $I$ is an $\I$--element of $\mathcal{Q}$.

By \prettyref{prop:finite-bi-is-semigrp-w-left-unit}, $\mathcal{Q}$ is a semigroup with $\sqbra{B_{\infty}}$ as its left unit, i.e.,
\[
\beta \sqbra{x} = \sqbra{Bx} = \sqbra{B_{\infty} x} = \sqbra{B_{\infty}} \sqbra{x} = \sqbra{x} \qquad\text{for all }x\in A_{\infty}.
\]
Therefore, $\beta$ is an identity map on $\mathcal{Q}$.

Now, suppose that $\mathcal{A}$ has an $\M$--element.
Choose an $M_{\infty} \in A_{\infty}$ such that $M_{\infty} \sim M$ by \prettyref{lem:Ainf-represents-A}.
Then $\sqbra{M_{\infty}}$ is an $\M$--element of $\mathcal{Q}$.
\end{proof}

\begin{corollary}\label{cor:eventual-image-fixed-point}
Let $\mathcal{A} = \paren{A, \ast, B, M}$ be a finite $\BM$--algebra.
For every $a \in A$, the set $L_{a}^{\infty}\sqbra{A}$ is precisely the set of all fixed points of $L_{a}$.
\end{corollary}

\begin{proof}
Clearly each fixed point of $L_{a}$ belongs to $L_{a}^{\infty}$.
It suffices to show that each point in $L_{a}^{\infty}$ is a fixed point of $L_{a}$.

Now, let $x \in L_{a}^{\infty}$.
By \prettyref{lem:Ainf-ideal-includes-stable-sets}, $x \in A_{\infty}$.
Since $L_{a}$ is a bijection on $L_{a}^{\infty} \subseteq A_{\infty}$, there exists a $y \in A_{\infty}$ such that
\[
x = L_{a}\paren{y} = ay.
\]
By \prettyref{lem:Ainf-represents-A}, Choose an $a_{\infty} \in A_{\infty}$ such that
\[
a_{\infty} \sim a.
\]
Since $\mathcal{Q}$ is an idempotent semigroup by \prettyref{thm:finite-bmi-is-idempotent},
\[
a_{\infty}a_{\infty} \sim a_{\infty}.
\]
The map $\beta$ is an identity on $\mathcal{Q}$, so 
\[
Ba_{\infty} \sim a_{\infty}.
\]
Using these equivalences, we obtain
\begin{align*}
ax &= a\paren{ay} \\
&= a \paren{a_{\infty} y} \\
&= a_{\infty} \paren{a_{\infty} y} \\
&= B a_{\infty} a_{\infty} y \\
&= a_{\infty}a_{\infty} y \\
&= a_{\infty} y \\
&= a y \\
&= x.
\end{align*}
Therefore, $x$ is a fixed point of $L_{a}$.
\end{proof}

\subsection{Non-standard fixed-point combinators}

It is well-known that there exists a closed term that is not a fixed-point combinator but indistinguishable from fixed-point combinators using their B\"ohm trees.

\begin{example}[cf.~\citet{Statman_1993}]\label{exa:ns-fpc}
Consider the following closed term:
\[
T := \B\M\paren{\B\paren{\B\M}\B}
\]
For each $n\in\Nat$, let
\[
A_{n} := \paren{\B x}^{n} \paren{\B\M\paren{\B x}}.
\]
Then
\begin{align*}
Tx &= \B\M\paren{\B\paren{\B\M}\B}x \\
&\tow \M\paren{\B\paren{\B\M}\B x} \\
&\tow \M\paren{\B\M\paren{\B x}} = \M A_{0} \\
&\tow A_{0} A_{0} = \B\M\paren{\B x} A_{0} \\
&\tow \M \paren{\B x A_{0}} = \M A_{1}.
\end{align*}
and for every $n \geq 1$,
\begin{align*}
\M A_{n} &\tow A_{n} A_{n} = \paren{\B x}^{n} A_{0} A_{n} \\
&\tow^{n} x^{n} \paren{A_{0} A_{n}} = x^{n} \paren{\B\M\paren{\B x} A_{n}} \\
&\tow x^{n} \paren{\M\paren{\B x A_{n}}} = x^{n} \paren{\M A_{n+1}}.
\end{align*}
Hence, by induction on $n$, we have
\[
Tx \tow^{\ast} x^{1+2+\cdots+n} \paren{\M A_{n+1}}.
\]
\end{example}

\begin{definition}
A closed term $T$ is called a \emph{non-standard fixed-point combinator} if for every $r \in \Nat$, there exists a term $R_{r}$ with variable $x$ such that $Tx \eqw x^{r} R_{r}$.
\end{definition}

The following notion is an algebraic analogue of nonstandard-fixed point combinators:

\begin{definition}
An element $\xi$ of $A$ is called a \emph{stabiliser} if $\xi a \in L_{a}^{\infty}\sqbra{A}$ holds for all $a \in A$.
\end{definition}

\begin{theorem}\label{thm:ns-fpc-is-stabiliser}
Let $\mathcal{A} = \paren{A, \ast, B, M}$ be a $\BM$--algebra and let $T$ be a non-standard fixed-point combinator.
Then the $\mathcal{A}$--interpretation $T^{\mathcal{A}}$ is a stabiliser.
\end{theorem}

\begin{proof}
Let $a \in A$.
For each $r\geq 1$, there exists a term $R_{r}$ with variable $x$ such that $Tx \eqw x^{r} R_{r}$.
Applying the substitution $\sqbra{a/x}$ on both sides, we have
\[
T^{\mathcal{A}} a = \sqbra{a/x}\paren{Tx}^{\mathcal{A}} = \sqbra{a/x}\paren{x^{r} R_{r}}^{\mathcal{A}} = L_{a}^{r}\paren{\sqbra{a/x} R_{r}^{\mathcal{A}}}.
\]
Thus
\[
T^{\mathcal{A}} a \in \bigcap_{r \in \Nat} L_{a}^{r}\sqbra{A}. \qedhere
\]
\end{proof}

In the finite case, the eventual image $L_{a}$ coincides with the set of all fixed points.
Thus we obtain the following result:

\begin{theorem}\label{thm:ns-fpc-is-y}
Let $\mathcal{A} = \paren{A, \ast, B, M}$ be a finite $\BM$--algebra and let $T$ be a non-standard fixed-point combinator.
Then the $\mathcal{A}$--interpretation $T^{\mathcal{A}}$ is a $\Y$--element.
\end{theorem}

\begin{proof}
Immediate from \prettyref{thm:ns-fpc-is-stabiliser} and \prettyref{cor:eventual-image-fixed-point}.
\end{proof}

This gives an affirmative answer to \prettyref{prob:finite-algebraic-Smullyan-problem}.

\begin{corollary}\label{cor:finite-bm-has-y}
Every finite $\BM$--algebra has a $\Y$--element.
\end{corollary}

Our proof also demonstrates that an infinite-valued invariant is required to distinguish non-standard fixed-point combinators from true fixed-point combinators.

\begin{corollary}
The variety of $\BM$--algebras is not generated by the class of finite $\BM$--algebras.
Here, `generated' means obtained by applying the operations of taking direct products, subalgebras, and homomorphic images.
\end{corollary}

\begin{proof}
Let $T := \B\M\paren{\B\paren{\B\M}\B}$.
The equation
\[
T x = x \paren{T x}
\]
holds for all finite $\BM$--algebras (\prettyref{thm:ns-fpc-is-y}), whereas it does not hold for certain infinite $\BM$--algebras (\prettyref{thm:infinite-bmi-wo-y}).
By Birkhoff's variety theorem (\citet{Birkhoff_1935}), the class of $\BM$--algebras is not generated by the class of finite $\BM$--algebras.
\end{proof}

\section{Related open problems}\label{sec:open-problems}

We conclude the paper by presenting several open problems related to \prettyref{prob:Smullyan-problem} and \prettyref{prob:algebraic-Smullyan-problem}.

\subsection{Permutative effects}

Existing research on Smullyan's problem has focused on the role of \emph{permutative effects} (\citet{MW_1991,Statman_2011}).
For a combinator $\mathbf{X}$ with a reduction rule
\[
\mathbf{X} x_{1} \cdots x_{n} \tow t,
\]
a `permutative effect' refers to a case where, for some $i < j$, an occurrence of $x_{j}$ precedes an occurrence of $x_{j}$ in the right-hand side $t$.
For example, the $\C$--combinator,
\[
\C xyz \tow xzy,
\]
has a permutative effect on $y$ and $z$, whereas $\B$, $\M$ and $\I$ do not.

The absence of permutative effects in the reduction rules, however, does not imply the absence of a fixed-point combinator.
To see this, consider the $\W$--combinator, whose reduction rule is
\[
\W xy \tow xyy.
\]
Like $\B$, $\M$ and $\I$, the $\W$--combinator has no permutative effect.
Nevertheless, the $\BWI$--fragment has a fixed-point combinator.
Indeed, Curry's fixed-point combinator
\[
Y := \lambda x. \paren{\lambda y. x\paren{yy}} \paren{\lambda y. x\paren{yy}}
\]
can be represented in the $\BWI$--fragment as follows:
\begin{align*}
Y &= \lambda x. \paren{\lambda y. x\paren{yy}} \paren{\lambda y. x\paren{yy}} \\
&\rightsquigarrow \lambda x. \paren{\lambda y. \B xyy} \paren{\lambda y. \B xyy} \\
&\rightsquigarrow \lambda x. \paren{\B\W\B x} \paren{\B\W\B x} \\
&\rightsquigarrow \lambda x. \paren{\W\I} \paren{\B\W\B x} \\
&\rightsquigarrow \B \paren{\W\I} \paren{\B\W\B} =: Y_{\BWI}.
\end{align*}
The resulting $Y_{\BWI}$ peforms as intended:
\begin{align*}
Y_{\BWI} x &\tow \paren{\W\I} \paren{\B\W\B x} \\
&\tow \I \paren{\B\W\B x} \paren{\B\W\B x} \\
&\tow \paren{\B\W\B x}\paren{\B\W\B x} \\
&\tow \W\paren{\B x}\paren{\B\W\B x} \\
&\tow \paren{\B x}\paren{\B\W\B x}\paren{\B\W\B x} \\
&\tow x \paren{\paren{\B\W\B x} \paren{\B\W\B x}} \eqw x \paren{Y_{\BWI} x}.
\end{align*}
Note that this fixed-point combinator does \emph{not} satisfy $Y_{\BWI} x \tow^{\ast} x\paren{Y_{\BWI}x}$.

Thus, the $\BWI$--fragment has a fixed-point combinator, whereas, as we have seen, the $\BMI$--fragment does not.
None of these fragments possesses permutative effects.
Therefore, the absence of a permutative effect is not the root cause for the absence of a fixed-point combinator within the $\BMI$--fragment.
This comparison leads to the following question:

\begin{problem}
What properties of the reduction rules determine whether or not a fixed-point combinator exists?
\end{problem}

More specifically, what differences between the reduction rules $\M x \tow xx$ and $\W xy \tow xyy$ account for the different answers in these two fragments?
Consider a reduction rule of the form $\mathbf{X} x_{1} \cdots x_{n} \tow t$, where $t$ has no constant symbols.
The term on the right-hand side can be viewed as a labelled binary tree with the natural numbers $1, \ldots, n$ as its leaves (see \prettyref{fig:trees-of-M-and-W} for $\M$ and $\W$).
Is it possible to describe the root cause for the absence of a fixed-point combinator in terms of this tree?
These questions are not yet properly formulated and should be understood as \emph{koans} in the sense used by \citet{Barendregt_2020}.

\begin{figure}
\begin{minipage}[b]{0.45\linewidth}
    \centering
    \begin{tikzpicture}[
        baseline=(root.base),
        level distance=7mm,
        sibling distance=12mm,
        every node/.style={inner sep=1pt}
    ]
    \node (root) {$\bullet$}
        child {node {$1$}}
        child {node {$1$}};
    \end{tikzpicture}
    \subcaption{$\M x_{1} \tow x_{1}x_{1}$.}
\end{minipage}
\begin{minipage}[b]{0.45\linewidth}
    \centering
    \begin{tikzpicture}[
        baseline=(root.base),
        level distance=7mm,
        sibling distance=12mm,
        every node/.style={inner sep=1pt}
    ]
    \node (root) {$\bullet$}
        child {node {$\bullet$}
            child {node {$1$}}
            child {node {$2$}}
        }
        child {node {$2$}};
    \end{tikzpicture}
    \subcaption{$\W x_{1}x_{2} \tow x_{1}x_{2}x_{2}$.}
\end{minipage}

\caption{Labelled binary trees for $\M$ and $\W$.\label{fig:trees-of-M-and-W}}
\end{figure}
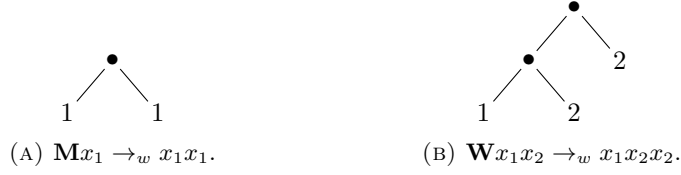

\subsection{Turing-type fixed-point combinators}

Consider Turing's fixed-point combinator in lambda calculus:
\[
\Theta := \theta\theta,\qquad \text{where}\ \theta := \lambda x. \lambda y.y\paren{xxy}.
\]
This closed $\lambda$--term possesses the following stronger fixed-point property:
\[
\Theta x = \theta\theta x \to_{\beta}^{\ast} x\paren{\theta\theta x} = x\paren{\Theta x}.
\]
The $\lambda$--abstraction of this $\lambda$--term can be eliminated in full combinatory logic as follows:
\begin{align*}
\theta &= \lambda x. \lambda y.y\paren{xxy} \\
&\rightsquigarrow \lambda x. \S\I\paren{xx} \\
&\rightsquigarrow \lambda x. \S\I\paren{\W\I x} \\
&\rightsquigarrow \B\paren{\S\I}\paren{\W\I}, \\
\Theta &\rightsquigarrow \B\paren{\S\I}\paren{\W\I} \paren{\B\paren{\S\I}\paren{\W\I}}.
\end{align*}
The construction above uses the $\S$--combinator, whose reduction rule
\[
\S xyz \tow xz\paren{yz}.
\]
has a permutative effect on $y$ and $z$.
As previously mentioned, the $\BWI$--fragment does not have such a permutative effect.
We therefore ask whether this fragment admits
a Turing-type fixed-point combinator.

\begin{problem}\label{prob:Turing-type-fpc-in-BWI}
Is there a closed term $\Theta$ in the $\BWI$--fragment such that $\Theta x \tow^{\ast} x\paren{\Theta x}$ for a fresh variable symbol $x$?
\end{problem}

Note that the $\BMI$--version of this problem has been answered negatively by \citet{MW_1991,Statman_1993} prior to the solution to \prettyref{prob:Smullyan-problem}.

The preceding discussion raises a more general question about the existence of Turing-type fixed-point combinators.
As the $\BWI$--fragment shows, the absence of permutative effects does not imply the absence of a fixed-point combinator.
This observation, however, does not settle the corresponding question for the Turing-type fixed-point combinators.

\begin{problem}
What properties of the reduction rules determine whether or not a Turing-type fixed-point combinator exists?
In particular, does the absence of permutative effects imply the absence of such a combinator?
\end{problem}

One can also consider the (partially ordered) algebraic version of \prettyref{prob:Turing-type-fpc-in-BWI}.

\begin{definition}
Let $\mathcal{A} = \paren{A, \ast, \leq}$ be a partially ordered magma, i.e. a magma $\paren{A, \ast}$ with a partial order $\leq$ such that the map $\ast \colon A\times A\to A$ is monotonic.
An element $\xi$ of $A$ is called
\begin{enumerate}
\item a \emph{$\B$--element} if $\xi xyz \leq x\paren{yz}$ for all $x,y,z\in A$.
\item an \emph{$\W$--element} if $\xi x y \leq x y y$ for all $x\in A$.
\item an \emph{$\I$--element} if $\xi x \leq x$ for all $x \in A$.
\item a \emph{$\Theta$--element} if $\xi x \leq x\paren{\xi x}$ for all $x \in A$.
\end{enumerate}
A partially ordered magma eqipped with $B$--, $W$-- and $\I$--elements is called a \emph{partially ordered $\BWI$--algebra}.
\end{definition}

\begin{problem}[An algebraic version of \prettyref{prob:Turing-type-fpc-in-BWI}]
Does every partially ordered $\BWI$--algebra have a $\Theta$--element?
\end{problem}

\bibliographystyle{plainnat}
\bibliography{refs}

\appendix

\section{Code availability}\label{app:code-availability}

The source code for reproducing \prettyref{obs:sat-observation} is available at \url{https://github.com/hyperfinitism/magma-search} (commit ID: \texttt{5c18b05}).

\texttt{magma-search} is a command-line tool for searching for finite magmas that satisfy specified conditions.
As a condition, one can specify whether or not there are elements that satisfy a given set of equations.
This tool provides two search methods.
One performs an exhaustive search to enumerate all isomorphism classes of magmas that satisfy the given conditions.
The other uses a parallel SAT solver (MallobSat~\citep{SS_2024,SRS_2026} via RustSAT~\citep{Jabs_2025}) to determine whether such a magma exists and, if so, to output an example.

It should be noted that the results presented in this paper do not depend on any computational experiments.
For example, \prettyref{obs:sat-observation} is a special case of \prettyref{cor:finite-bm-has-y}, which was proved independently of computational experiments.

\end{document}